%% file: ammernet.tex
\documentclass[11pt, a4paper]{article}
\usepackage{amsmath}
\usepackage{graphics,graphicx}
\usepackage{verbatim}
\usepackage{amssymb}
\usepackage{latexsym}
\usepackage{amsmath, amsthm, amssymb}

\newtheorem{corollary}{Corollary}[section]

\newtheorem{lemma}[corollary]{Lemma}
\newtheorem{proposition}[corollary]{Proposition}
\newtheorem{remark}[corollary]{Remark}
\newtheorem{theorem}[corollary]{Theorem}
\newcommand{\mylabel}[1]{\label{#1}
            \ifx\undefined\stillediting
            \else \fbox{$#1$}\fi }
\newcommand{\BE}{\begin{equation}}
\newcommand{\BEQ}[1]{\BE\mylabel{#1}}
\newcommand{\EEQ}{\end{equation}}
\newcommand{\rfb}[1]{\mbox{\rm
   (\ref{#1})}\ifx\undefined\stillediting\else:\fbox{$#1$}\fi}

\newfont{\Blackboard}{msbm10 scaled 1200}
\newcommand{\bl}[1]{\mbox{\Blackboard #1}}
\newfont{\roma}{cmr10 scaled 1200}

\def\CC{\rm \hbox{C\kern-.56em\raise.4ex
         \hbox{$\scriptscriptstyle |$}\kern+0.5 em }}

\renewcommand{\cline}{{\bl C}}

\newcommand{\nline}  {{\bl N}}
\newcommand{\rline}  {{\bl R}}

\newcommand{\half}   {{\frac{1}{2}}}

\def\s{\sigma}

\def\cA{{\cal A}}

\def\cH{{\cal H}}

\newcommand{\mm}    {{\hbox{\hskip 0.5pt}}}

\newcommand{\bluff} {{\hbox{\raise 15pt \hbox{\mm}}}}

\newcommand{\FORALL} {{\hbox{$\hskip 11mm \forall \;$}}}

\makeatletter
\def\section{\@startsection {section}{1}{\z@}{-3.5ex plus -1ex minus
    -.2ex}{2.3ex plus .2ex}{\large\bf}}
\makeatother
\def\be{\begin{equation}}
\def\ee{\end{equation}}

\def\ds{\displaystyle}
\newcommand{\caD}{{\cal D}}

\newcommand{\caH}{{\cal H}}
\newcommand{\lb}{\beta}
\newcommand{\norm}[2]{\|#1 \| _{#2} }

\def \R {{\mathbb{R}}}
\def \N {{\mathbb{N}}}
\def \C {{\mathbb{C}}}

\begin{document}
\thispagestyle{empty}
\title{\bf Boundary feedback stabilization of a chain of serially connected strings}
\author{
Ka\"{\i}s Ammari
\thanks{UR Analysis and Control of Pde, UR 13ES64, Department of Mathematics, Faculty of Sciences of Monastir, University of Monastir, Tunisia, e-mail: kais.ammari@fsm.rnu.tn} 
 \, and \, Denis Mercier \thanks{UVHC, LAMAV, FR CNRS 2956, F-59313 Valenciennes, France, email: denis.mercier@univ-valenciennes.fr}}
\date{}
\maketitle
%
%
\begin{quotation}
{\bf Abstract.} {\small We consider $N$ strings connected to one another and forming a particular network which is a chain of strings. We study a stabilization problem and precisley we prove that the energy of the solutions of the dissipative system decay exponentially to zero when the time tends to infinity, independently of the densities of the strings. Our technique is based on a
frequency domain method and a special analysis for the resolvent. Moreover, 
by same appraoch, we study the transfert function associated to the chain of strings and the stability of the Schr\"odinger system.}
\end{quotation}
\section{Introduction} \label{secintro}

\setcounter{equation}{0}
We consider the evolution problem $(P)$  described by the following system of $N$ equations: 
\begin{equation*}
\leqno(P) 
\left \{
\begin{array}{l}
(\partial_t^2 u_{j}- \rho_j \partial_x^2u_{j})(t,x)=0,\, x\in(j,j+1),\, t\in(0,\infty),\, j = 0,...,N-1, \\
\rho_0 \, \partial_x u_0(t,0) = { \bf \partial_t u_0(t,0)},\ u_{N-1}(t,N)=0,\, t\in(0,\infty),\\
u_{j-1}(t,j)=u_{j}(t,j),\, t\in(0,\infty),\, j = 1,...,N-1,\\
- \rho_{j-1} \partial_x u_{j-1}(t,j)+ \rho_j \partial_x u_{j}(t,j)= 0,\, t\in(0,\infty),\, j = 1,...,N-1, \\
u_j(0,x)=u_j^0(x),\ \partial_t u_j(0,x)=u_j^1(x), \,x \in (j,j+1),\,   j=0,...,N-1,
\end{array}
\right.\\
\end{equation*}
 
where $\rho_j > 0, \, \forall \, j=0,...,N-1$.

We can rewrite the system (P) as a first order hyperbolic system, by putting
$$
V_j = \left( 
\begin{array}{ll}
\partial_t u_j \\ \rho_j \, \partial_x u_j
\end{array}
\right), \; \hbox{and} \; V_j^0 = \left( 
\begin{array}{c}
u^1_j \\ \rho_j \, \partial_x u^0_j
\end{array}
\right), \, 0 \leq j \leq N-1,
$$

\begin{equation*}
\leqno(P^\prime) 
\left \{
\begin{array}{l}
(\partial_t V_{j}- B_j  \, \partial_x V_{j})(t,x)=0,\, x\in(j,j+1),\, t\in(0,\infty),\, j = 0,...,N-1, \\
C_0 V_0(t,0) = 0, \, C_{N-1} V_{N-1} (t,N) = 0,\, t\in(0,\infty),\\
V_{j-1}(t,j)=V_{j}(t,j),\, t\in(0,\infty),\, j = 1,...,N-1,\\
V_j(0,x)= V_j^0 (x), \,x\in(j,j+1),\,   j=0,...,N-1,
\end{array}
\right.\\
\end{equation*}
where 
\be
\label{C0CN}
B_j = \left(\begin{array}{ll} 0 & 1 \\ \rho_j & 0\end{array} \right), C_0 = \left(\begin{array}{lclc} 1 & - 1 \\ 0 & 0 \end{array} \right), \, C_{N-1} = \left(\begin{array}{ll} 1 & 0 \\ 0 & 0 \end{array} \right), \, 0 \leq j \leq N -1.
\ee

Models of the transient behavior of some or all of the state variables describing the motion of flexible structures have been of great interest
in recent years, for details about physical motivation for the models, see \cite{dagerzuazua}, \cite{lagnese} and the references therein. Mathematical
analysis of transmission partial differential equations is detailed in \cite{lagnese}.

Let us first introduce some notation and definitions which will be used throughout the rest of the paper, in particular some which are linked to the notion of $C^{\nu }$- networks,
$\nu \in \nline$ (as introduced in \cite{jvb}). \\
Let $\Gamma$ be a connected topological graph embedded in $\rline$, with $N$ edges ($N \in \nline^{*}$).  
Let $K=\{k_{j}\, :\, 0\leq j\leq N-1\}$ be the set of the edges of $\Gamma$. Each edge $k_{j}$ is a Jordan curve in $\rline$ and is assumed to be parametrized by its arc length $x_{j}$ such that
the parametrization $\pi _{j}\, :\, [j,j+1]\rightarrow k_{j}\, :\, x_{j}\mapsto \pi _{j}(x_{j})$ is $\nu$-times differentiable, i.e. $\pi _{j}\in C^{\nu }([j,j+1],\rline)$ for all $0\leq j\leq N-1$. The density of the edge $k_j$ is $\rho_j>0$. 
The $C^{\nu}$- network $R$  associated with $\Gamma$ is then defined as the union $$R=\bigcup _{j=0}^{N-1}k_{j}.$$

We study a feedback stabilization problem for a wave and a Schr\"dinger equations in networks, see \cite{ammari1}-\cite{amjellk}, \cite{lagnese} and Figure \ref{fig}. 
\begin{figure}[ht]
\begin{center}
\input{maillagecordes}
\caption{Serially connected strings}
\end{center}
\label{fig}
\end{figure}
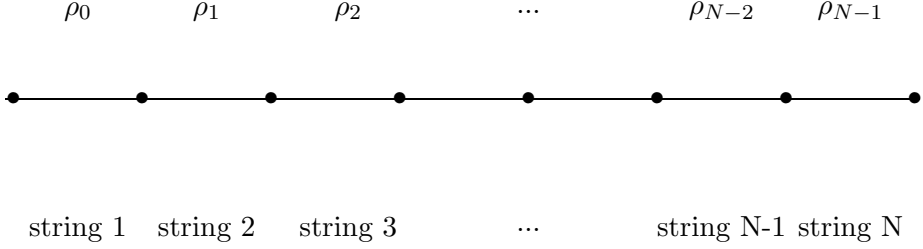

More precisely, we study a linear system modelling the vibrations of a chain of strings. For each edge $k_{j}$, the scalar function $u_j(t,x)$ for $x \in R$ and $t > 0$
contains the information on the vertical displacement of the string, $0 \leq j \leq N-1$. 

Our aim is to study the behaviour of the resolvent of the spatial operator
which is defined in Section \ref{resolvent} and to obtain stability result for $(P)$.

We define the natural energy $E(t)$ of a solution $\underline{u} = (u_0,...,u_{N-1})$ of $(P)$ and the natural energy of a solution $V$ of $(P^\prime)$, respectively, by
\be 
\label{energy1}
E(t)=\frac{1}{2} \ds \sum_{j=0}^{N-1} \left( \int_{j}^{j+1} \left(|\partial_t u_{j}(t,x)|^2+ \rho_j \, |\partial_x u_{j}(t,x)|^2\right){\rm d}x \right),
\ee
\be 
\label{energy1b}
e(t)=\frac{1}{2} \, \ds \sum_{j=0}^{N-1} \left\|V_{j}\right\|^2_{L^2_{\rho_j}(j,+,j+1) \times L^2(j,j+1)}, \, 0 \leq j \leq N-1,
\ee
where $L^2_{\rho_j} (j,j+1) = L^2_{\rho_j} ((j,j+1),dx) = L^2((j,+j+1),\rho_j \, dx)$.

We note that $E(t)  \asymp e(t), \, \forall \, t \geq 0$.

We can easily check that every sufficiently smooth solution of $(P)$ satisfies the following dissipation law 
\begin{equation}\label{dissipae1}
E^\prime(t) = - \ds \bigl|\partial_t u_{0}(t,0)\bigr|^2\leq 0, \, 
e^\prime(t) = - \ds \bigl|C_0 V_{0}(t,0)\bigr|^2\leq 0,
\end{equation}
and therefore, the energy is a nonincreasing function of the time variable $t$.

The result concerns the well-posedness of the solutions of $(P)$ and the exponential decay of the energy $E(t)$ of the solutions of $(P)$. 

\medskip

The main result of this paper then concerns the precise asymptotic behaviour of the solutions of $(P)$. 
Our technique is based on a frequency domain method and a special analysis for the resolvent.

\medskip

This paper is organized as follows:
In Section \ref{wellposed}, we give the proper functional setting for system $(P)$ and prove that the system is well-posed.
In Section \ref{resolvent}, we then show that the energie of system $(P)$ tends to zero.
We study, in Section \ref{resolvent}, the stabilization result for $(P)$ by the frequency domain technique 
and give the explicit decay rate of the energy of the solutions of $(P)$. Finally, in the last sections, 
we study the transfert function associated to a string network and the exponential stability of the Schr\"odinger system.

\section{\label{wellposed}Well-posedness of the system}

In order to study system $(P)$ we need a proper functional setting. 
We define the following spaces  
$$
H = \ds \prod_{j=0}^{N-1} (L^2_{\rho_j}(j,j+1) \times L^2(j,j+1))
$$
and
$$
V= \bigg \{\underline{u}=(u_0,...,u_{N-1}) \in \ds \prod_{j=0}^{N-1} H^1(j,j+1),\\ 
u_{N-1}(N)=0, \, u_{j-1}(j) = u_{j} (j), \, j=1, \ldots ,N-1 \bigg \},
$$
equipped with the inner products
\begin{equation}\label{ipVb}
<V,\tilde{V}>_{H}=\sum_{j=0}^{N-1}\left(\int_j^{j+1} \rho_j \,  u_{j}(x) \,  \overline{\tilde{u}_{j}(x)} + v_{j}(x) \,  \overline{\tilde{v}_{j}(x)} 
dx\right), \, V = \left(\begin{array}{c} \underline{u} \\ \underline{v} \end{array} \right), 
\tilde{V} = \left(\begin{array}{c} \underline{\tilde{u}} \\ \underline{\tilde{v}} \end{array} \right),
\end{equation}
\begin{equation}\label{ipV}
<\underline{u},\,\underline{\tilde{u}}>_{V}=\sum_{j=0}^{N-1}\left(\int_j^{j+1} \rho_j \partial_x u_{j}(x)\partial_x  \overline{\tilde{u}_{j}(x)} 
dx\right).
\end{equation}

It is well-known that system $(P)$ may be rewritten as the first order evolution equation
\begin{equation} \left\{
\begin{array}{l}
U^\prime =\mathcal{A} U,\\
U(0)=(\underline{u}^{0},\,\underline{u}^{1})=U_0,\end{array}\right.\label{pbfirstorder}\end{equation}
where $U$ is the vector $U=(\underbar{u},\,\partial_t \underbar{u})^t$ and the operator $\mathcal{A} : {\cal D}({\cal A}) \rightarrow {\cal H} = V\times \ds \prod_{j=0}^{N-1} L^2(j,j+1)$ is defined by 
\[\mathcal{A} (\underline{u},\underline{v})^t:=(\underline{v}, (\rho_j \partial_x^2u_{j})_{0 \leq j\leq N-1})^t,\] 
with
\begin{multline*}
{\cal D}({\cal A}):=\left\{(\underline{u},\,\underline{v})\in \prod_{j=0}^{N-1} H^2(j,j+1) \times V \,:
\mbox {\textrm{satisfies }} \,(\ref {e2}) \; \mbox{\textrm{to}} \; (\ref {e5}) \; \mbox{\textrm{hereafter}}
\right\},\end{multline*}
\begin{equation}\label{e2}
\rho_0 \, \partial_x u_{0}(0) =  v_{0}(0)
\end{equation}
\begin{equation}\label{e5}
- \rho_j \partial_x u_{j}(j) + \rho_{j-1} \partial_x u_{j-1}(j)= 0, \quad j = 1,...,N-1. 
\end{equation}

It is clear that $\mathcal{H}$ is a Hilbert space,
equipped with the usual inner product 
\begin{multline*}
\left\langle\left(\begin{array}{c}\underline{u}\\\underline{v}\end{array}\right),
\left(\begin{array}{c}\underline{\tilde{u}}\\ \underline{\tilde{v}}\end{array}\right)\right\rangle_{{\cal H}} = 
\sum_{j=0}^{N-1}\left(\int_{j}^{j+1}\left(v_{j}(x)\overline{\tilde{v}_{j}(x)}
+ \rho_j \partial_x u_{j}(x)\partial_x\overline{\tilde{u}_{j}(x)}\right){\rm d}x. \right.
\end{multline*} 

By the same way we define the operator $A$ as following:
$$
A : {\cal D}(A) \subset H \rightarrow H, A V = B \, \partial_x \, V, \, \forall \,  V \in {\cal D}(A),
$$ 
where $${\cal D}(A) = \left\{V = (V_j)_{0 \leq j \leq N-1} \in H, \, V_j \in (H^1(j,j+1))^2, \,  V_{j-1} (j) = V_j (j), 1 \leq j \leq N-1, \right.
$$
$$
C_0 V_0 (0 ) = 0, \, C_{N-1} V_{N-1} (N) = 0  $$ and 
$B = (B_j)_{0 \leq j \leq N-1}.$

Now we can prove the well-posedness of system $(P)$ and that the solution of $(P)$ satisfies the dissipation law (\ref{dissipae1}).

\begin{proposition}\label{3exist1} 
(i) For an initial datum $U_{0}\in \mathcal{H}$, there exists a unique solution $U\in C([0,\,+\infty),\, \mathcal{H})$
to  problem (\ref{pbfirstorder}). Moreover, if $U_{0}\in \mathcal{D}(\mathcal{A})$, then
$$U\in C([0,\,+\infty),\, \mathcal{D}(\mathcal{A}))\cap C^{1}([0,\,+\infty),\, \mathcal{H}).$$

(ii) The solution $\underline{u}$ of $(P)$ with initial datum in $\mathcal{D}(\mathcal{A})$ satisfies \rfb{dissipae1}.
Therefore the energy is decreasing.
\end{proposition}

\begin{proof}
(i) By Lumer-Phillips' theorem (see \cite{Pazy, tucsnakbook}), it suffices to show that
$\mathcal{A}$ is dissipative and  maximal.

We first prove that $\mathcal{A}$ is dissipative. Take $U=(\underline{u},\underline{v})^{t}\in \mathcal{D}(\mathcal{A})$. 
Then
\begin{multline*}
\left\langle\mathcal{A}U,\, U \right\rangle_{\mathcal{H}}=\sum_{j=0}^{N-1}\left(\int_{j}^{j+1}\left(\rho_j \partial_x^2u_{j}(x)\overline{v_{j}(x)}
+ \rho_j \partial_x v_{j}(x)\partial_x \overline{u_{j}(x)}\right){\rm d}x\right ).
\end{multline*}
By integration by partsand by using the transmission and boundary conditions, we have

\begin{equation}\label{dissipativeness}
\Re\left(\left\langle\mathcal{A}U,\, U \right\rangle_{\mathcal{H}}\right)=- \left|v_{0}(0)\right|^2\leq 0.
\end{equation}
This shows the dissipativeness of $\mathcal{A}$.

\medskip

Let us now prove that $\mathcal{A}$ is maximal, i.e. that
$\lambda I-\mathcal{A}$ is surjective for some $\lambda>0$.

Let $(\underline{f}, \underline{g})^{t}\in \mathcal{H}$. We look for $U=(\underline{u}, \underline{v})^{t}\in \mathcal{D}(\mathcal{A})$ solution of 
\begin{equation}\label{eqmaxmon}
(\lambda I-\mathcal{A})\left(\begin{array}{c}
\underline{u}\\\underline{v}\end{array}\right)=\left(\begin{array}{c}
\underline{f}\\\underline{g}\end{array}\right),\end{equation} 
or equivalently
\begin{equation} \left\{
\begin{array}{ll}
\lambda u_{j}-v_{j}=f_{j} & \forall j\in\{0,...,N-1\},\\
\lambda v_{j}- \rho_j \partial^{2}_xu_{j}=g_{j} & \forall j\in\{0,...,N-1\}.\end{array}\right.\label{eqmaxmon2}\end{equation}

Suppose that we have found $\underline{u}$ with the appropriate regularity. 
Then for all  $j\in\{0,...,N-1\},$ we have
\begin{equation} v_{j}:=\lambda u_{j}-f_{j}\in V.\label{maxmonv}\end{equation}
It remains to find $\underline{u}$. By (\ref{eqmaxmon2}) and (\ref{maxmonv}), $u_{j}$ must
satisfy, for all $j=0,...,N-1$,
$$ 
\lambda^{2}u_{j}- \rho_j \partial^{2}_xu_{j}=g_{j}+\lambda f_{j}.
$$ 
Multiplying these identities by a test function $\underline{\phi}$, integrating in space and using integration by
parts, we obtain
\begin{multline*}
\sum_{j=0}^{N-1}\int_j^{j+1}
\left(\lambda^2u_{j}\overline{\phi_{j}}+ \rho_j \partial_x u_{j}\partial_x \overline{\phi_{j}}\right)dx
-\sum_{j=0}^{N-1}\left[\rho_j \partial_xu_{j}\overline{\phi_{j}}\right]_j^{j+1}  \\
=\sum_{j=0}^{N-1}\int_j^{j+1}\left( g_j+\lambda f_j \right)\overline{\phi_j} \, dx.
\end{multline*}
Since $(\underline{u},\underline{v})\in \mathcal{D}(\mathcal{A})$ and $(\underline{u},\underline{v})$ satisfies (\ref{maxmonv}), we then have
\begin{multline}\label{maxmoneq1}
\sum_{j=0}^{N-1}\int_j^{j+1}\left(\lambda^2u_{j}\overline{\phi_{j}}+\rho_j \partial_xu_{j}\partial_x\overline{\phi_{j}}\right)dx + 
\\
\left(\lambda u_0(0) - f_0(0) \right) \overline{\phi_0}(0) =\sum_{j=0}^{N-1}\int_j^{j+1}\left(g_j+\lambda f_j\right)\overline{\phi_j}dx.
\end{multline}
This problem has a unique solution $\underline{u}\in V$ by Lax-Milgram's lemma, because the left-hand side  of (\ref{maxmoneq1}) is coercive on $V$.
If we consider $\underline{\phi}\in \ds \prod_{j=0}^{N-1}\mathcal{D}(j,j+1)\subset V$, then $\underline{u}$
satisfies
$$\begin{array}{c}
\displaystyle{\lambda^{2}u_{j}-\rho_j \partial_x^{2}u_{j}=g_{j}+\lambda f_{j} \quad\hbox{ in } \mathcal{D}^\prime (j,j+1),\quad  j=0,\cdots,N-1.} 
\end{array}$$
This directly implies that  $\underline{u}\in \ds \prod_{j=0}^{N-1}H^{2}(j,j+1)$ and then 
$\underline{u}\in V\cap \ds \prod_{j=0}^{N-1}H^{2}(j,j+1)$. 
Coming back to (\ref{maxmoneq1}) and by integrating by parts, we find
$$\begin{array}{ll}
- \ds \sum_{j=0}^{N-1}\left(\rho_j \partial_x u_{j}(j) \overline{\phi_{j}}(j) - \rho_{j} \partial_x u_{j}(j+1)\overline{\phi_{j}}(j+1) \right)\\
+ \left(\lambda u_0(0) - f_0(0) \right)\overline{\phi_0}(0)
= 0.
\end{array}
$$
Consequently, by taking particular test functions  $\underline{\phi}$, we obtain 
$$\begin{array}{c}
\rho_0 \partial_x u_{0}(0)= v_0(0) \quad \hbox{and}\quad \rho_j \partial_x u_{j}(j) - \rho_{j-1} \partial_x u_{j-1}(j) =0, \, j=1,\cdots,N-1.
\end{array}$$ 
In summary we have found $(\underline{u},\underline{v})^{t}\in \mathcal{D}(\mathcal{A})$ satisfying (\ref{eqmaxmon}), which finishes the proof of (i).

\medskip

(ii) To prove (ii), it suffices to derivate the energy (\ref{energy1}) for regular solutions and to use system $(P)$.
The calculations are analogous to those of the proof of the dissipativeness of $\mathcal{A}$ in (i), and then, are left to the reader.
\end{proof}

\begin{remark} \label{opha}
By the same we can prove that the operator $A$ is a m-dissipatif operator of $H$ and generates a $C_0-$ semigroup of contractions of $H$.
\end{remark}

\bigskip   

%
%
 
\section{Exponential stability} \label{resolvent}

We prove a decay result of the energy of system $(P)$, independently of $N$ and of the densities, for all initial data in the energy space. Our technique is based on a frequency domain method and a special analysis for the resolvent. 
\begin{theorem} \label{lr}
There exists a constant $C, \omega >0$ such that, for all $(\underline{u}^0,\underline{u}^1)\in {\cal H}$, the solution of system $(P)$ satisfies the following estimate
\BEQ{EXPDECEXP3nb}
E(t) \le C \, e^{- \omega \,t} \, \left\Vert (\underline{u}^0,\underline{u}^1) \right\Vert_{{\cal H}}^2,
\FORALL t > 0.
\EEQ
\end{theorem}
  
{\it Proof.}
By classical result (see Huang \cite{huang} and Pr\"{u}ss \cite{pruss}) it suffices to show that ${\cal A}$ satisfies the
following two conditions:) of a $C_0$ semigroup
of contractions on a Hilbert space:

\be 
\rho ({\cal A})\supset \bigr\{i \beta \bigm|\beta \in \rline \bigr\} \equiv i \rline, \label{1.8w} \ee 
and \be \limsup_{|\beta |\to \infty }  \|(i\beta -{\cal A})^{-1}\|_{{\cal L}(\caH)} <\infty, 
\label{1.9} 
\ee 
where $\rho({\cal A})$ denotes the resolvent set of the operator ${\cal A}$.

Then the proof of Theorem \ref{lr} is based on the following two lemmas.

\begin{lemma} \label{condsp}
The spectrum of ${\cal A}$ contains no point on the imaginary axis.
\end{lemma}

\begin{proof}
Since ${\cal A}$ has compact resolvent,
its spectrum $\s({\cal A})$ only consists of eigenvalues of ${\cal A}$. We will
show that the equation
\be 
{\cal A} Z = i \beta \, Z
\label{1.10}
\ee
with $Z= \left(\begin{array}{l} \underline{y} \cr \underline{v} \end{array} \right) \in {\cal D}({\cal A})$ and $\beta \neq 0$ has only the trivial solution.

By taking the
inner product of (\ref{1.10}) with $Z \in {\cal H}$ and using
\be
\label{1.7}
\Re <{\cal A}Z,Z>_{{\cal H}} = - \, \left| v_0 (0)\right|^2,
\ee
we obtain
that $v_0(0)=0$. Next, we eliminate $\underline{v}$ in (\ref{1.10})
to get
a second order ordinary differential system:
\be
\left\{ \begin{array}{l} 
\rho_j \, \frac{d^2 y_j}{dx^2} + \beta^2 \, y_j = 0, \, (j,j+1), \, j = 0,...,N-1,\\
y_0(0) = \frac{d y_0}{dx}(0)=0, \, y_{N-1} (N) = 0, \\
y_{j-1}(j) = y_{j}(j),\, j = 1,...,N-1.
\end{array} 
\right.
\label{1.11}
\ee
The above system has only trivial solution.

\end{proof}

\begin{lemma}\label{lemresolvent}
The resolvent operator of $\mathcal{A}$ satisfies condition \rfb{1.9}.
\end{lemma}

\begin{proof}
In order to prove \rfb{1.9} or by equivalence the following
\be 
\limsup_{|\beta |\to \infty }  \|(i\beta -A)^{-1}\|_{{\cal L}(H)} <\infty, 
\label{1.9n} 
\ee 
we  will compute and estimate the resolvent of the operator $A$ associated to the problem $(P^\prime)$. \\
More precisely, let $\lambda=i\lb,\lb \in \R,$  $G=(G_0,...,G_{N-1}) \in H,$ we look for $W=(W_0,...,W_{N-1})\in \caD(A)$  solution of
\be \label{respb} (i \lb - B \partial_x ) W =G, \ee
where $B=(B_0,...,B_{N-1})$ and $$B_j=\left(\begin{array}{ll}
0 & 1  \\ 
\rho_j & 0 
\end{array} \right), \;j=0,...N-1.
$$
We want to prove that there exists a constant $C$ independent of $\lb$ such that 
\be \label{estresolv} \norm{W}{H} \leq C \norm{G}{H}\ee

{\bf First step :} Computation of the resolvent\\
From (\ref{respb}) we have 
$$\partial_x W_j = i\lb B_j^{-1} W_i - B_j^{-1}G_j,j=0,...,N-1$$
therefore 

\be\label{W0} W_0(x)=e^{i \lb (x-1) B_0^{-1} }F_0-\int_1^x e^{i \lb (x-s) B_0^{-1} } B_0^{-1} G_0(s)\; ds, \forall x\in [0,1],\ee
and
\be \label{Wj} W_j(x)=e^{i \lb (x-j) B_j^{-1} }F_j-\int_j^x e^{i \lb (x-s) B_j^{-1} } B_j^{-1} G_j(s) \; ds, \forall x\in [j,j+1],j=1,2,...,N-1,\ee
where $F_0=W_0(1)$ and $F_j=W_j(j), j=0,...,N-1.$
For simplification we set 
\be\label{Gjt} \tilde{G_j} (x)=\int_j^x e^{i \lb(x-s)  B_j^{-1} } B_j^{-1} G_j(s) \; ds ,j=0,...,N-1.\ee
Using the transmission conditions at nodes $j=1,...,N-1$ we have 
\be\label{F1} F_1=F_0, \mbox{ and } F_j=W_{j-1}(j), j=2,...,N-1,\ee
which implies that 
\be\label{Fj0} F_j=\left( \prod_{k=j-1}^{k=1} e^{i \lb B_k^{-1}}\right)F_0- \left(\sum_{p=2}^{j-1}\left( \prod_{k=j-1}^{k=p} e^{i \lb B_k^{-1}}\right) 
\tilde{G}_{p-1}(p)
+ \tilde{G}_{j-1}(j)\right) ,j=2,...,N-1.\ee
For all $j=1,...,N-1$ we set $$M_j(\lb)=\left( \prod_{k=j-1}^{k=1} e^{i \lb B_k^{-1}}\right)$$ and 
\be \label{gammaj}\Gamma_j(\lb)=\sum_{p=2}^{j-1}\left( \prod_{k=j-1}^{k=p} e^{i \lb B_k^{-1}}\right) \tilde{G}_{p-1}(p)
 +\tilde{G}_{j-1}(j),\ee
hence  \be \label{Fj} F_j=M_j(\lb) F_0-\Gamma_j(\lb).\ee
Note that the solution $W$ is completely determined if $F_0$ is known. Indeed, it suffices to insert the identity (\ref{Fj}) in  (\ref{Wj}). 

Thus, we give the equation satisfied by  $F_0.$ The boundary conditions at nodes $x=0$ and $x=N$ are respectively $C_0 W_0(0)=\left(\begin{array}{l}
0  \\ 
0 
\end{array} \right),$ and 
$C_{N-1} W_{N-1}(N)=\left(\begin{array}{l}
0  \\ 
0 
\end{array} \right),$
 where $C_0, C_{N-1}$ are  the matrices  given in (\ref{C0CN}). 
Since the second lines of $C_0$ and $C_{N-1}$ vanish the previous equations may be written as 
$$(1,-1). W_0(0)=0,\;  \mbox{\rm and }  (1,0). W_{N-1}(N)=0,$$
where "$.$" represents the matrix-product of a vector line by a vector column. 
These equations are  equivalent to
\be\label{eq1} (1,-1).e^{-i\lb B_0^{-1} }F_0 = (1,-1),\tilde{G_0}(0) \ee
and 
$$(1,0).e^{i\lb B_{N-1}^{-1} }F_{N-1} = (1,0).\tilde{G_{N-1}}(N). $$
Inserting (\ref{Fj}) in the previous equation we get 
\be\label{eq2}(1,0).e^{i\lb B_{N-1}^{-1} }M_{N-1}(\lb)F_0 = (1,0).\left( \tilde{G}_{N-1}(N)+e^{i\lb B_{N-1}^{-1} } \Gamma_{N-1}(\lb)\right). \ee
If we denote by $H_{N-1}$ the $2\times 2$ matrix whose the first line is the vector line is $(1,-1).e^{-i\lb B_0^{-1} }$ 
and the second line is $(1,0).e^{i\lb B_{N-1}^{-1} }M_{N-1}(\lb)$ i.e 

\be\label{HN} H_{N-1}=\left(\begin{array}{ll}
(1,-1).e^{-i\lb B_0^{-1} }\\ 
(1,\;\;0).e^{i\lb B_{N-1}^{-1}}
\end{array} \right)\ee
and $Y_{N-1}$ the $2\times 1$ vectors columns by

\be \label{YN} Y_{N-1}=\left(\begin{array}{c}
(1,-1).\tilde{G_0}(0)\\ 
(1,0).\left( \tilde{G_{N-1}}(N)+e^{i\lb B_{N-1}^{-1} } \Gamma_{N-1}(\lb)\right)
\end{array} \right)\ee
then equation (\ref{eq1}) and (\ref{eq2}) are equivalent to the following system: 

\be \label{eq3}  H_{N-1} F_0= Y_{N-1}. \ee

{\bf Second  step: estimate of $F_0$} \\

We first start by given an estimation of $\tilde{G}=(\tilde{G_0},...,\tilde{G}_{N-1})$ where $\tilde{G_j}$ are defined in (\ref{Gjt}).

For all $j=0,...,N-1,$ the matrix is $B_j$ is invertible: 
$$B_j^{-1}=\left(
\begin{array}{cc}
 0 & \frac{1}{\rho_j} \\
 1 & 0
\end{array}
\right)$$

\ and we  easily find after some computation that 
\be\label{expbj} e^{i\lb x  B_j^{-1}}=\left(
\begin{array}{cc}
 \cos(\dfrac{\lb x}{\sqrt{\rho_j}}) & \dfrac{i \sin(\dfrac{\lb x}{\sqrt{\rho_j}})}{\sqrt{\rho_j}} \\
 i \sqrt{\rho_j}\sin(\dfrac{\lb x}{\sqrt{\rho_j}}) & \cos(\dfrac{\lb x}{\sqrt{\rho_j}})
\end{array}
\right).
\ee
Since $\lb \in \R,$ from the previous identity, we directly get the following estimates 

\be \label{est1}   
|\tilde{G_j}(j)| 
\lesssim \norm{G}{H},\;|\tilde{G_j}(j+1)| 
\lesssim \norm{G}{H}, j=0,...,N-1 \ \ee

\be \label{est2}   
\norm{\tilde{G}}{H}  
\lesssim \norm{G}{H}. \ee

From  the definition of $\Gamma_j$ in (\ref{gammaj}) we also get 
\be \label{est3} |\Gamma_j(\lb) | \lesssim  \norm{G}{H}, j=1,...,N-1. \ee

It follows that 
\be \label{est4} \|Y_{N-1}\| \lesssim \norm{G}{H}. \ee

Note that from (\ref{expbj}) the entries of $H_{N-1}$ are bounded and so it is for the entries of the matrix $( \mbox{\rm com} H_{N-1})^T.$  Assume for the moment that 
there exists a constant $\gamma_{N-1}>0$ such that 
\be \label{det} \forall \lb \in\R, |\det(H_{N-1})|\geq \gamma_{N-1},\ee 
then it follows with (\ref{est4}) that 
\be\label{estF0} \|F_0\|=\dfrac{1} {\det H_{N-1}} ( \mbox{\rm com} H_{N-1})^T Y_{N-1} \lesssim  \norm{G}{\caH}.\ee
It remains to prove (\ref{det}).


The idea of the proof is that (\ref{det}) is well known for $N=1$ and that this property spreads by iteration. 

First, similarly to (\ref{HN}) we define for all $N\in \N^*$ the matrix 
\be\label{HNt} \tilde{H}_{N-1}=\left(\begin{array}{ll}
(1,-1).e^{-i\lb B_0^{-1} }\\ 
(0,\;\;1).e^{i\lb B_{N-1}^{-1}}
\end{array} \right)\ee 
and we set 
$$D_{N-1}=\det(H_{N-1}),\;\; \tilde{D}_{N-1}=\det(\tilde{H}_{N-1}), \forall N\in \N^*.$$
Particularly, for $N=1$ we have 
$$H_0=\left(
\begin{array}{cc}
 \cos(\dfrac{\lb}{\sqrt{\rho_0}})+i \sqrt{\rho_0} \sin(\dfrac{\lb}{\sqrt{\rho_0}}) & -\cos(\dfrac{\lb}{\sqrt{\rho_0}})-i\dfrac{ \sin(\dfrac{\lb}{\sqrt{\rho_0}})}{\sqrt{\rho_0}} \\
 \cos(\dfrac{\lb}{\sqrt{\rho_0}}) & -i\dfrac{ \sin(\dfrac{\lb}{\sqrt{\rho_0}})}{\sqrt{\rho_0}}
\end{array}
\right),$$
and 
$$\tilde{H}_0=\left(
\begin{array}{cc}
 \cos(\dfrac{\lb}{\sqrt{\rho_0}})+i \sqrt{\rho_0} \sin(\dfrac{\lb}{\sqrt{\rho_0}}) & -\cos(\dfrac{\lb}{\sqrt{\rho_0}})-i\dfrac{ \sin(\dfrac{\lb}{\sqrt{\rho_0}})}{\sqrt{\rho_0}} \\
i \sqrt{\rho_0} \sin(\dfrac{\lb}{\sqrt{\rho_0}}) &  \cos(\dfrac{\lb}{\sqrt{\rho_0}})
\end{array}
\right).$$
Thus $$D_0= \cos(\dfrac{2\lb}{\sqrt{\rho_0}}) +  i \dfrac{\sin(\dfrac{2\lb}{\sqrt{\rho_0}})}{\sqrt{\rho_0}},\; \tilde{D}_0=
\cos(\dfrac{2\lb}{\sqrt{\rho_0}}) +  i \sqrt{\rho_0} \sin(\dfrac{2\lb}{\sqrt{\rho_0}}),$$
and  $$|D_0|^2=\cos^2(\dfrac{2\lb}{\sqrt{\rho_0}})+\dfrac{\sin^2(\dfrac{2\lb}{\sqrt{\rho_0}})}{\rho_0}\geq \min(1,\dfrac{1}{\rho_0})>0,$$
 $$|\tilde{D}_0|^2=\cos^2(\dfrac{2\lb}{\sqrt{\rho_0}})+\rho_0\sin^2(\dfrac{2\lb}{\sqrt{\rho_0}})\geq \min(1,\rho_0)>0.$$
It is useful for the sequel to remark that 
$$\Re(D_0 \overline{\tilde{D}_0)}=1.$$  
Using (\ref{expbj}) we have the following identity 
$$\left\{\begin{array}{lll}
D_{N-1} =  \cos(\dfrac{\lb}{\sqrt{\rho_{N-1}}})D_{N-2}+\dfrac{i}{\sqrt{\rho_{N-1}}} \sin(\dfrac{\lb }{\sqrt{\rho_{N-1}}})\tilde{D}_{N-2}, \\ 
\tilde{D}_{N-1} = i  \sqrt{\rho_{N-1}} \sin(\dfrac{\lb }{\sqrt{\rho_{N-1}}})D_{N-2}+\cos(\dfrac{\lb}{\sqrt{\rho_{N-1}}})\tilde{D}_{N-2}.
\end{array} \right.
$$
A simple computation shows that 
$$\Re(D_{N-1}\overline{\tilde{D_{N-1}}})=\Re(D_{N-2}\overline{\tilde{D}_{N-2}})$$ consequently,
$$\forall N \in \N^*, \;\Re(D_{N-1}\overline{\tilde{D}_{N-1}})=1.$$

Now, since
$$|D_{N-1}|^2=
$$
$$
(\cos(\dfrac{\lb}{\sqrt{\rho_{N-1}}}),\sin(\dfrac{\lb}{\sqrt{\rho_{N-1}}}))
\left(\begin{array}{ll}
|D_{N-2}|^2 &\dfrac{1}{\sqrt{\rho_{N-1}}} \Im(D_{N-2}\overline{\tilde{D}_{N-2}}) \\ 
\dfrac{1}{\sqrt{\rho_{N-1}}}\Im(D_{N-2}\overline{\tilde{D}_{N-2}}) & |\dfrac{1}{\sqrt{\rho_{N-1}}}\tilde{D}_{N-2}|^2
\end{array} \right)
\left(\begin{array}{l}
\cos(\dfrac{\lb}{\sqrt{\rho_{N-1}}}) \\ 
\sin(\dfrac{\lb}{\sqrt{\rho_{N-1}}})
\end{array} \right),
$$
it follows that  
\be \label{mu}|D_{N-1}|^2 \geq \mu_{min,N-2},\ee
 where $\mu_{min,N-2}$ is the smallest eigenvalue of the matrix in the previous identity.
The determinant of this matrix is: 
$$\dfrac{1}{\rho_{N-1}}\Re(D_{N-2}\overline{\tilde{D}_{N-2}})^2=\dfrac{1}{\rho_{N-1}}.$$
Since $D_{N-2}$ and $\tilde{D}_{N-2}$ are clearly bounded the trace of this matrix is bounded, i.e 
$$\exists C'_{N-1}>0, |D_{N-2}|^2+ |\dfrac{1}{\sqrt{\rho_{N-1}}}\tilde{D}_{N-2}|^2 \leq C'_{N-1}.$$ 
It follows that $$\mu_{min,N-2}\geq \dfrac{1}{\rho_{N-1} C'_{N-1}}.$$ 
Setting $C_{N-1}=\sqrt{\dfrac{1}{\rho_{N-1} C'_{N-1}}},$ then (\ref{mu}) implies (\ref{det}). Consequently we have prove the estimate (\ref{estF0}) for $F_0.$

Finally, using estimates (\ref{estF0}), (\ref{est1}), (\ref{est2}) in (\ref{F1}) and (\ref{Fj0}) and the fact that the matrices involved in (\ref{Fj0}) are uniformly bounded we get 
$$\|F_j\| \lesssim  \norm{G}{\caH}, \forall j=1,...N-1.$$
Using (\ref{estF0}) and the previous estimates in (\ref{W0}) and (\ref{Wj}), we get (\ref{estresolv}).

Which implies \rfb{1.9n} and thereafter \rfb{1.9}, and end the proof of Theorem \ref{lr}.

\end{proof}

\section{Comments and related questions}
The same strategy can be applied to stabilize the following models and to verify and compute the transfer function.
 
\subsection{Transfer function}
We can use the same strategy to verify that the operator $H(\lambda) = \lambda \, C^*(\lambda^2 I +\underline{A}  )^{-1} \, C \in {\cal L}(U), \, \lambda \in \cline_+,$ satisfies the property (\ref{Hest}) of the following lemma,
where here $\underline{A}$ is the self-adjoint operator corresponding to the conservative problem associated  to problem (P), namely we replace in (P) the boundary feedback condition by 
\begin{equation}\label{newe2}
\rho_0 \, \partial_x u_{0}(t,0) =0,
\end{equation} 
i.e., 

$\underline{A} : {\cal D}(\underline{A}) \subset \underline{H} = \ds \prod_{j=0}^{N-1} L^2(j,j+1) \rightarrow \underline{H}$ is defined by 
\[\underline{A} (\underline{u}):=(- \, \rho_j \partial_x^2u_{j})_{0 \leq j\leq N-1},
\] 
with
\begin{multline*}
{\cal D}(\underline{A}):=\left\{\underline{u} \in \prod_{j=0}^{N-1} H^2(j,j+1)\,:
\mbox {\textrm{satisfies }} \,(\ref {e2bn}) \; \mbox{\textrm{to}} \; (\ref {e5bn}) \; \mbox{\textrm{hereafter}}
\right\},\end{multline*}
\begin{equation}\label{e2bn}
\rho_0 \, \partial_x u_{0}(0) =  0
\end{equation}
\begin{equation}\label{e5bn}
- \rho_j \partial_x u_{j}(j) + \rho_{j-1} \partial_x u_{j-1}(j)= 0, \quad j = 1,...,N-1. 
\end{equation}
$$
C \in {\cal L}(\mathbb{C}, V^\prime = {\cal D}(\underline{A}^\half)^\prime), \, Ck = \sqrt{\rho_0} \, \underline{A}_{-1} \underline{N} k = k \, \left( \begin{array}{c} \frac{1}{\sqrt{\rho_0}} \,  \delta_0 \\ . \\ . \\. \\ 0 \end{array} \right), \, \forall \, k \in \mathbb{C}, 
$$
$$
C^* \underline{u} = \left( \frac{1}{\sqrt{\rho_0}} \, u_0(0) \ 0 \; ... \; 0 \right), \, \forall \, \underline{u} \in V,
$$
where $\underline{A}_{-1}$ is the extension of $\underline{A}$ to 
$({\cal D}(\underline{A}))^\prime$ (the duality is in the sense of $\underline{H}$) and $\underline{N}$ is the Neumann map, 
$$
\left\{
\begin{array}{lll}
\rho_j \, \partial_{x}^2 (\underline{N}k)_j = 0, \, (j,j+1), \, 0 \leq j \leq N -1, \\ \partial_{x} (\underline{N} k)_0 (0) = k, \, (\underline{N}k)_{N -1} (N) = 0,\\
\rho_{j-1} \, \partial_{x} (\underline{N}k)_{j-1} (j) = \rho_{j} \, \partial_{x} (\underline{N}k)_{j} (j), \, 1 \leq j \leq N-1.
\end{array}
\right.
$$


\begin{lemma} The transfer function $H$ satisfies the following estimate:   
\be\label{Hest}
\sup_{\Re \lambda = \gamma} \left\| \lambda \, C^*(\lambda^2 I 
+\underline{A})^{-1} \, C \right\|_{{\cal L}(U)} < \infty,
\ee
for $\gamma > 0.$
\end{lemma}
\begin{proof}

In the same way as the proof of Lemma \ref{lemresolvent}, we give an equivalent formulation of the function $H.$ For that purpose, we consider  $W=(W_j)_{0 \leq j \leq N-1} \in H,W_j \in (H^1(j,j+1))^2$ solution of
\be \label{respbb}  
\begin{array}{l}
(\lambda - B \partial_x ) W =0, \\
W_{j-1} (j) = W_j (j), 1 \leq j \leq N-1, \\
\tilde{C}_0 W_0 (0 ) =  \left(\begin{array}{l}
z \\
0
\end{array} \right)
,\, C_{N-1} W_{N-1} (N) = 0  
\end{array}
\ee where $z \in \cline,$ $B$ is defined as in the proof of Lemma \ref{lemresolvent}, $C_{N-1}$ is the matrix given in (\ref{C0CN}) and 
$\tilde{C}_0  =\left(\begin{array}{ll}
0 &1  \\ 
0&0
\end{array} \right).$
Therefore, for $\lambda \in \cline, \Re(\lambda)=\gamma>0,$  the transfer function is
$$ H(\lambda) : z\in \cline \mapsto (1,0). W_0(0)\in \cline.$$ 
Consequently to prove (\ref{Hest}) it suffices to check that for  a fixed  $\gamma>0,$ there exists a constant $c_\gamma>0$ such that 
\be \label{448} \forall \lambda \in \cline, \Re(\lambda)= \gamma,\;\; |(1,0). W_0(0)| \leq c_\gamma |z| . 
\ee

Using (\ref{W0}), (\ref{Wj}), (\ref{Fj0}), we 
have 
$$W_0(x)=e^{\lambda (x-1) B_0^{-1}} F_0, W_j(x)=e^{\lambda (x-j)B_j^{-1}} F_j,j =1,...,N-1,$$
where 
$$ F_0=W_0(1),\; F_1=F_0,\; \mbox{\rm and } F_j=\left( \prod_{k=j-1}^{k=1} e^{\lambda B_k^{-1}}\right)F_0, \; j=2,...,N-1. 
$$
Therefore, from the boundary conditions at $x=0$ and $x=N,$ we find that $F_0$ is the solution of 
$$ H_{N-1}F_0= \left(\begin{array}{l}
z \\
0
\end{array} \right),$$
where 
$H_{N-1}$ is the $2\times2 $ matrix
$$\left(\begin{array}{ll}
(0,1).e^{-\lambda B_0^{-1}} \\ 
(1,0 ). \left(\displaystyle \prod_{k=N-1}^{k=1} e^{\lambda B_k^{-1}}\right)
\end{array} \right), 
$$
with the convention that $\left(\displaystyle \prod_{k=N-1}^{k=1} e^{\lambda B_k^{-1}}\right)
$ is the identity matrix if $N=1.$ 

{\bf Estimate of $F_0$}\\
Since for all $j$ $$e^{\lambda  x B_j^{-1} }=\left(
\begin{array}{cc}
 \cosh(\frac{\lambda x}{c_j}) & \frac{\sinh(\frac{ \lambda x}{c_j})}{c_j} \\
 c_j \sinh(\frac{\lambda x }{c_j}) & \cosh(\frac{ \lambda x}{c_j})
\end{array}
\right),
$$
it is clear that there exists a constant $c'_\gamma>0$ sucht that 
\be \label{up} \forall \lambda \in \cline, \Re(\lambda)=\gamma,\; \|H_{N-1}\|\leq c'_\gamma.\ee
We need a similar estimate for $H_{N-1}^{-1};$ this will be done by giving a lower uniform bound of $|D_{N-1}|$ on the line $\Re(\lambda)=\gamma,$ where 
we have set $D_{N-1}=\det(H_{N-1}).$
Thus we introduce the matrix
$$\tilde{H}_{N-1}= \left(\begin{array}{ll}
(0,1).e^{-\lambda B_0^{-1}} \\ 
(0,1). \left( \prod_{k=N-1}^{k=1} e^{\lambda B_k^{-1}}\right)
\end{array} \right), \; N\geq 1,
$$
and set $\tilde{D}_{N-1}=\det(\tilde{H}_{N-1}).$ 
Now, we prove by iteration that
$$ \Re (D_{N-1} \overline{\tilde{D}_{N-1}})\geq k_{N-1} >0.$$   
For $N=1$ we have 
$$H_0=\left(
\begin{array}{cc}
-\sinh(\frac{\lambda }{c_0}) &\cosh(\frac{\lambda }{c_0}) \\
 1 & 0
\end{array}
\right),\; 
\tilde{H}_0=\left(
\begin{array}{cc}
-\sinh(\frac{\lambda }{c_0}) &\cosh(\frac{\lambda }{c_0}) \\
 0 & 1
\end{array}
\right),
$$
thus $$\Re (D_0 \overline{\tilde{D}_0})=\dfrac{1}{2} \sinh(\frac{2 \gamma}{c_0} )>0.$$ 
Assume that there exists a constant $k_{N-2}>0$ such that
\be \label{rere} \Re (D_{N-2} \overline{\tilde{D}_{N-2}})\geq k_{N-2} >0.\ee  

We have the following easily checked identities 
$$D_{N-1}= \cosh(\frac{\lambda}{c_{N-1}}) D_{N-2}+\dfrac{1}{c_{N-1}}\sinh(\frac{\lambda}{c_{N-1}})\tilde{D}_{N-2} $$
$$\tilde{D}_{N-1}=c_{N-1} \sinh(\frac{\lambda}{c_{N-1}}) D_{N-2}+\cosh(\frac{\lambda}{c_{N-1}})\tilde{D}_{N-2}. $$
A computation leads to 
$$\begin{array}{lll}
\Re (D_{N-1} \overline{\tilde{D}_{N-1}})&=&\Re(\cosh(\frac{\lambda}{c_{N-1}})\overline{\sinh(\frac{\lambda}{c_{N-1}})}) 
(c_{N-1}|D_{N-2}|^2+\dfrac{1}{c_{N-1}}|\overline{\tilde{D}_{N-2}}|^2)\\
&+&(|\cosh(\frac{\lambda}{c_{N-1}})|^2+|\sinh(\frac{\lambda}{c_{N-1}})|^2)\Re (D_{N-2} \overline{\tilde{D}_{N-2}})\\
&=& \dfrac{1}{2} \sinh(\frac{2 \gamma}{c_{N-1}} )(c_{N-1}|D_{N-2}|^2+\dfrac{1}{c_{N-1}}|\overline{\tilde{D}_{N-2}}|^2)\\
&+&\cosh(\frac{2 \gamma}{c_{N-1}} )\Re (D_{N-2} \overline{\tilde{D}_{N-2}})\\
&\geq& \cosh(\frac{2 \gamma}{c_{N-1}} ) k_{N-2}\\
&=& k_{N-1}>0. 
\end{array} 
$$
We have proved (\ref{rere}). It follows 
$$|D_{N-1} \tilde{D}_{N-1}|\geq k_{N-1}>0.$$ But $|\tilde{D}_{N-1}|$  is obviously upper bounded on the line $\Re(\lambda)=\gamma,$ consequently there exists a constant $k'_{N-1}>0$ such that 
$$\forall \lambda, \Re(\lambda)=\gamma,\; |D_{N-1}|\geq k'_{N-1}>0.$$
Finally, with (\ref{up}) we deduce that $H_{N-1}^{-1}$ is bounded on the line $\Re(\lambda)=\gamma$ and it follows that there exits $c_\gamma >0$ such that 
$$\forall z \in \cline, \forall \lambda :  \Re(\lambda)=\gamma,\; |F_0|\leq c_\gamma z.$$

Conclusion: (\ref{448}) is a direct consequence of the previous estimate. The proof is complete 

\end{proof}
As application is that the open loop system associated to $(P)$ is satisfies a regularity property. 
\begin{corollary}
Let $T > 0$. Then, for all $v \in L^2(0,T)$ the following problem
\begin{equation*}
\left \{
\begin{array}{l}
(\partial_t^2 \psi_{j}- \rho_j \partial_x^2 \psi_{j})(t,x)=0,\, x\in(j,j+1),\, t\in(0,\infty),\, j = 0,...,N-1, \\
\rho_0 \, \partial_x \psi_0(t,0) = v(t),\ \psi_{N-1}(t,N)=0,\, t\in(0,\infty),\\
\psi_{j-1}(t,j)=\psi_{j}(t,j),\, t\in(0,\infty),\, j = 1,...,N-1,\\
- \rho_{j-1} \partial_x \psi_{j-1}(t,j)+ \rho_j \partial_x \psi_{j}(t,j)= 0,\, t\in(0,\infty),\, j = 1,...,N-1, \\
\psi_j(0,x)=0,\ \partial_t \psi_j(0,x)=0, \,x \in (j,j+1),\,   j=0,...,N-1.
\end{array}
\right.
\end{equation*}
admits a unique solution $(\psi,\partial_t \psi) \in C(0,T; {\cal H})$ which satisfies the following regularity property (says also open loop admissibility): there exists a constant $C > 0$ such that 
$$
\int_0^T \left| \partial_t \psi_0 (t,0)\right|^2 \, dt \leq C \, \left\|v\right\|^2_{L^2(0,T)}, \, \forall \, v \in L^2(0,T).
$$
\end{corollary}

Moreover, according to \cite[Theorem 2.2]{ammari}, we have that: 

\begin{corollary}
The system $(P)$ is exponentially stable in the energy space if and only if there exists $T, C > $ such that 
$$
\int_0^T \left| \partial_t \varphi_0 (t,0)\right|^2 \, dt \geq C \, \left\|(\varphi^0,\varphi^1)\right\|^2_{{\cal H}}, \, \forall \, (\varphi^0,\varphi^1) \in {\cal D}({\cal A}),
$$
where 
\begin{multline*}
{\cal D}({\cal A}):=\left\{(\underline{u},\,\underline{v})\in \prod_{j=0}^{N-1} H^2(j,j+1) \times V \,:
\mbox {\textrm{satisfies }} \,(\ref {e2sd}) \; \mbox{\textrm{to}} \; (\ref {e5sd}) \; \mbox{\textrm{hereafter}}
\right\},\end{multline*}
\begin{equation}\label{e2sd}
\rho_0 \, \partial_x u_{0}(0) =  0
\end{equation}
\begin{equation}\label{e5sd}
- \rho_j \partial_x u_{j}(j) + \rho_{j-1} \partial_x u_{j-1}(j)= 0, \quad j = 1,...,N-1. 
\end{equation}
and
$\varphi = (\varphi_0,...,\varphi_{N-1})$ satisfies the following problem
\begin{equation*}
\left \{
\begin{array}{l}
(\partial_t^2 \varphi_{j}- \rho_j \partial_x^2 \varphi_{j})(t,x)=0,\, x\in(j,j+1),\, t\in(0,\infty),\, j = 0,...,N-1, \\
\rho_0 \, \partial_x \varphi_0(t,0) = 0,\ \varphi_{N-1}(t,N)=0,\, t\in(0,\infty),\\
\varphi_{j-1}(t,j)=\varphi_{j}(t,j),\, t\in(0,\infty),\, j = 1,...,N-1,\\
- \rho_{j-1} \partial_x \varphi_{j-1}(t,j)+ \rho_j \partial_x \varphi_{j}(t,j)= 0,\, t\in(0,\infty),\, j = 1,...,N-1, \\
\varphi_j(0,x)=\varphi_j^0(x),\ \partial_t \varphi_j(0,x)=\varphi_j^1(x), \,x \in (j,j+1),\,   j=0,...,N-1.
\end{array}
\right.
\end{equation*}
\end{corollary}

\section{Schr\"{o}dinger system} \label{Schr}

\setcounter{equation}{0}
We consider the evolution problem $(S)$  described by the following system of $N$ equations: 
\begin{equation*}
\leqno(S) 
\left \{
\begin{array}{l}
(\partial_t u_{j}+i \rho_j \partial_x^2u_{j})(t,x)=0,\, x\in(j,j+1),\, t\in(0,\infty),\, j = 0,...,N-1, \\
\rho_0 \, \partial_x u_0(t,0) = { \bf i u_0(t,0)},\ u_{N-1}(t,N)=0,\, t\in(0,\infty),\\
u_{j-1}(t,j)=u_{j}(t,j),\, t\in(0,\infty),\, j = 1,...,N-1,\\
- \rho_{j-1} \partial_x u_{j-1}(t,j)+ \rho_j \partial_x u_{j}(t,j)= 0,\, t\in(0,\infty),\, j = 1,...,N-1, \\
u_j(0,x)=u_j^0(x),\,   j=0,...,N-1,
\end{array}
\right.\\
\end{equation*}
 
where $\rho_j > 0, \, \forall \, j=0,...,N-1$.

We define the natural energy $E(t)$ of a solution $u = (u_0,...,u_{N-1})$ of $(S)$ by
\be 
\label{schrenergy1}
E(t)=\frac{1}{2} \ds \sum_{j=0}^{N-1}  \int_{j}^{j+1} | u_{j}(t,x)|^2 {\rm d}x,
\ee
We can easily check that every sufficiently smooth solution of $(S)$ satisfies the following dissipation law 
\begin{equation}\label{schrdissipae1}
E^\prime(t) = - \ds \bigl| u_{0}(t,0)\bigr|^2\leq 0, \, 
\end{equation}
and therefore, the energy is a nonincreasing function of the time variable $t$.

In order to study system $(S)$ we introduce the following Hilbert space

$$
\mathcal{H}= \bigg \{u=(u_0,...,u_{N-1}) \in  \ds \prod_{j=0}^{N-1} L^2(j,j+1)) \bigg \},
$$
equipped with the inner product
$$
<u,\tilde{u}>_{\mathcal{H}}=\sum_{j=0}^{N-1} \int_j^{j+1}  u_{j}(x) \, \overline{\tilde{u}_{j}(x)} dx.
$$

The system $(S)$ is a first order evolution equation which as the form 
\begin{equation} \left\{
\begin{array}{l}
u' =\mathcal{A} u,\\
u(0)=u^{0},\,\end{array}\right.\label{schrfirstorder}\end{equation}

where $u^{0}=(u_0^0,u_1^0,...,u_{N-1}^0) \in \cH $  and the operator 
$\mathcal{A} : {\cal D}({\cal A}) \rightarrow \mathcal{H}$ is defined by 
\[\mathcal{A} u:=(-i \rho_j \partial_x^2u_{j})_{0 \leq j\leq N-1},\] 
with
\begin{multline*}
{\cal D}({\cal A}):=\left\{u\in \prod_{j=0}^{N-1} H^2(j,j+1) \,:
\mbox {\textrm{satisfies }} \,(\ref {schre2}) \; \mbox{\textrm{to}} \; (\ref {schre4}) \; \mbox{\textrm{hereafter}}
\right\},\end{multline*}
\begin{equation}\label{schre2}
\rho_0 \, \partial_x u_0(0) = { \bf i u_0(0)}, u_{N-1}(N)=0,
\end{equation}
\begin{equation}\label{schre3}
u_{j-1}(j)=u_{j}(j),\, j = 1,...,N-1,\\
\end{equation}
\begin{equation}\label{schre4}
-\rho_{j-1} \partial_x u_{j-1}(j)+ \rho_j \partial_x u_{j}(j)= 0,\, j = 1,...,N-1. \\
\end{equation}

Now we can prove the well-posedness of system $(S)$ and that the solution of $(S)$ satisfies the dissipation law (\ref{schrdissipae1}).

\begin{proposition}\label{schr3exist1} 
(i) For an initial datum $u^{0}\in \mathcal{H}$, there exists a unique solution $u\in C([0,\,+\infty),\, \mathcal{H})$
to  problem (\ref{schrfirstorder}). Moreover, if $u^{0}\in \mathcal{D}(\mathcal{A})$, then
$$u\in C([0,\,+\infty),\, \mathcal{D}(\mathcal{A}))\cap C^{1}([0,\,+\infty),\, \mathcal{H}).$$

(ii) The solution $u$ of $(S)$ with initial datum in $\mathcal{D}(\mathcal{A})$ satisfies \rfb{schrdissipae1}.
Therefore the energy is decreasing.
\end{proposition}

\begin{proof}
(i) By Lumer-Phillips' theorem, it suffices to show that
$\mathcal{A}$ is dissipative and  maximal.

$\mathcal{A}$ is clearly dissipative. Indeed, by integration by parts and by using the transmission and boundary conditions, we have

\begin{equation}\label{schrdissipativeness}
\forall  u \in \mathcal{D}(\mathcal{A}),\; \Re\left(\left\langle\mathcal{A}u,\, u \right\rangle_{\mathcal{H}}\right)=- \left|u_{0}(0)\right|^2\leq 0.
\end{equation}

\medskip

Now, let $f\in \mathcal{H}$. We look for $u\in \mathcal{D}(\mathcal{A})$ solution of 
\begin{equation}\label{schrbijection1}
-\mathcal{A} u = f.
\end{equation} 
or equivalently
\begin{equation}
i \rho_j \partial_x^2u_{j}=f_{j} ,\;  \forall j\in\{0,...,N-1\},
\label{schrbijection2}\end{equation}
and $u$ satisfies the boundary and transmission conditions (\ref{schre2})-(\ref{schre4}). 

The general solution of (\ref{schrbijection2}) is   
$$u_j(x)=\dfrac{1}{i\rho_j} \int_{j}^{x}\left( \int_{j}^{u} f_j(s) ds \right ) du  + p_i(x),\; j=0,...,N-1,\; x\in[j,j+1],$$ 
where each  $p_j$ is a polynomial of degree 1. It remains to find $p_j,\;j=0,...,N-1$ such that the equations  (\ref{schre2})-(\ref{schre4}) are satisfied. It is equivalent to solve a linear system with $2N$ equations and $2N$ unknowns.  This system admits an unique solution if and only if the corresponding homogeneous system admits only the trivial solution. 

So we suppose that $p_j, j=0,...,N-1$ are polynomials of degree 1 and satisfy (\ref{schre2})-(\ref{schre4}). Then by integrations by parts and using (\ref{schre2})-(\ref{schre4}) we get  
$$
\ds \sum_{j=0}^{N-1} \rho_j \int_j^{j+1} |p_j'(x)|^2 dx=  -\sum_{j=0}^{N-1} \rho_j \int_j^{j+1} p_j''(x) \overline{p_j(x) }dx=0.
$$
Consequently, the polynomials $p_j$ are constant and finally vanish from the continuity conditions and the 
right Dirichlet condition. 
Hence we have proved that (\ref{schrbijection1}) admits an unique solution. Consequently $\mathcal{A}$ is maximal.   

\medskip

(ii) To prove (ii), we use the same argument as in the proof of Theorem \ref{3exist1}.   
\end{proof}

\bigskip   

%
%
 
\subsection{Exponential stability of the Schr\"{o}dinger system} \label{scrresolvent}

The stability result of system $(S)$ is given by 
\begin{theorem} \label{schlr}
There exist constants  $C>0$ and $\omega >0$ such that, for all $u^0\in {\cal H}$, the solution of system $(S)$ satisfies the following estimate
\BEQ{schrEXPDECEXP3nb}
E(t) \le C \, e^{- \omega \,t} \, \left\Vert u^0 \right\Vert_{{\cal H}}^2,
\FORALL t > 0.
\EEQ
\end{theorem}
  
{\it Proof.} As in the proof of Theorem (\ref{lr}) the result is based on the following two lemmas.

\begin{lemma} \label{schrcondsp}
The spectrum of ${\cal A}$ contains no point on the imaginary axis.
\end{lemma}

\begin{proof}
Since ${\cal A}$ has compact resolvent,
its spectrum $\s({\cal A})$ only consists of eigenvalues of ${\cal A}$. We will
show that the equation
\be 
{\cal A} u  = i \beta \,u
\label{schr1.9nn}
\ee
with $u = (u_0,...,u_{N-1}) \in {\cal D}({\cal A})$ and $\beta \neq 0, \beta \in \R$ has only the trivial solution.

By taking the
inner product of (\ref{schr1.9nn}) with $u  \in {\cal H}$ and using (\ref{schrdissipativeness}) we get that $u_0(0)=0$. From the left boundary condition we deduce also that $\partial_x u_0(0)=0.$ Therefore we get that $u_0=0$ since $\rho_0 \partial^2_x u_0=\beta u_0.$ Therefore by iteration we easily find $u_j=0, j=1,...,N-1.$ 

The system (\ref{schr1.9nn})  has only trivial solution.

\end{proof}

\begin{lemma}\label{schrlemresolvent}
The resolvent operator of $\mathcal{A}$ satisfies 
\be 
\limsup_{|\beta |\to \infty }  \|(i\beta -\mathcal{A})^{-1}\|_{{\cal L}(H)} <\infty, 
\label{schr1.9n} 
\ee 
\end{lemma}

\begin{proof}
In order to prove \rfb{schr1.9n}  we look for $u=(u_0,...,u_{N-1})\in \caD(\cA)$  solution of
\be\label{schrL2} (i \beta -  \cA ) u =g,\ee  where $\beta \in \R$ and $g=(g_0,...,g_{N-1}) \in \mathcal{H}.$

We will consider two cases since for each case the method is different. 

{\bf First case : $\beta>0$. }

{\bf First step :} Computation of the resolvent

The solution of( \ref{schrL2}) satisfies  
\be \label{schrequation} i \beta u_j +i \rho_i \partial_x^2  u_j =g_j,\; i=0,...,N-1.\ee An easy calculation shows that 
\be\label{defuj} u_j(x)=G_i(x), + c_j^1 \cos (\dfrac{\sqrt{\beta }x}{\sqrt{\rho_j}}) + c_j^2\dfrac{1}{\beta} \sin (\dfrac{\sqrt{\beta }x}{\sqrt{\rho_j}}), x\in[j,j+1],\; j=0,...,N-1,\ee
where 
\be \label{defGJ}G_j(x)=\int_j^{x} \dfrac{\sin (\dfrac{\sqrt{\beta }(x-s)}{\sqrt{\rho_j}})}{ i\sqrt{\beta} \sqrt{\rho_j}}
 g_i(s) ds,\ee
and $c_j^1,c_j^2, \in \C.$ Note that $c_j^1=u_j(j)$ and $\rho_j (\partial_x u_j)(j)=c_j^2.$

Now, let $F_j=\left(
\begin{array}{l}
c_j^1\\ 
c_j^2\\
\end{array} \right),\;
j=0,...,N-1,$ 
then using the transmission conditions (\ref{schre3})-(\ref{schre4}) we have   
$$\begin{array}{lll}
F_{j+1}&=&\left(
\begin{array}{c}
u_j(j+1)\\ 
\rho_j(\partial_x u_j)(j+1)\\
\end{array} \right)\\
&=&
\left(\begin{array}{cc}
\cos(\dfrac{\sqrt{\beta}}{\sqrt{\rho_j}}) &\dfrac{\sin(\dfrac{\sqrt{\beta}}{\sqrt{\rho_j}})}{\sqrt{\beta}\sqrt{\rho_j}}  \\ 
-\sqrt{\beta}\sqrt{\rho_j}\sin(\dfrac{\sqrt{\beta}}{\sqrt{\rho_j}}) & \cos(\dfrac{\sqrt{\beta}}{\sqrt{\rho_j}}) \\
\end{array} \right)F_j+
\left(
\begin{array}{l}
G_j(j+1)\\ 
(\partial_x G_j)(j+1)\\
\end{array} \right).
\end{array} 
$$

For simplification we introduce the matrix $M_j$ and the vector $W_j$ as 
\be\label{defMJWJ} M_j=\left(\begin{array}{cc}
\cos(\dfrac{\sqrt{\beta}}{\sqrt{\rho_j}}) &\dfrac{\sin(\dfrac{\sqrt{\beta}}{\sqrt{\rho_j}})}{\sqrt{\beta}\sqrt{\rho_j}}  \\ 
-\sqrt{\beta}\sqrt{\rho_j}\sin(\dfrac{\sqrt{\beta}}{\sqrt{\rho_j}}) & \cos(\dfrac{\sqrt{\beta}}{\sqrt{\rho_j}}) \\
\end{array} \right),\; W_j=\left(
\begin{array}{l}
G_j(j+1)\\ 
(\partial_x G_j)(j+1)\\
\end{array} \right),\ee
hence the transmission conditions are
\be \label{defFj} F_{j+1}=M_j F_j +W_j, j=0,...,N-1.\ee

It follows that 
\be\label{schrfn} F_N=(\prod_{j=N-1}^0 M_j) F_0 +\sum_{k=0}^{N-2}(\prod_{j=N-2}^{k+1} M_j)W_k+W_{N-1}.\ee 
   
From the first boundary condition  (\ref{schre2}) we have $F_0=c_0^1\left(\begin{array}{c}
1 \\ 
i 
\end{array} \right), 
$ (i.e, $c_0^2=i c_0^1)$) therefore we now compute $c_0^1$ by using the second boundary condition  (\ref{schre2}). For that we set 

\be\label{schalphagamma} \prod_{j=N-1}^0 M_j=
\left(\begin{array}{cc}
\alpha_{N,1} & \gamma_{N,1} \\ 
\alpha_{N,2} & \gamma_{N,2}
\end{array} \right),\ee
and 

\be\label{schrsecondmembre} -\sum_{k=0}^{N-2}(\prod_{j=N-2}^{k+1} M_j)W_k+W_{N-1}=
\left(\begin{array}{c}
\omega_{N,1} \\ 
\omega_{N,2}
\end{array} \right).\ee
Thus we find that 
$$c_0^1=\dfrac{\omega_{N,1}}{\alpha_{N,1}+i \gamma_{N,1}}.$$

This last identity completely determine the solution $u$ of (\ref{schrL2}). 

{\bf Second  step :} Estimate of $F_0$ for $\beta$ large. 

From one hand, since the order of each matrix $M_j$ is 
\be\label{schorderM}\left(\begin{array}{cc}
O(1) & O(\dfrac{1}{\sqrt{\beta}} )\\ 
O(\sqrt{\beta}) & O(1)
\end{array} \right)
\ee then it is easy to see that all the matrices involved in (\ref{schrsecondmembre}) have the same order. 

On the other hand, from (\ref{defGJ})  we have clearly
\be\label {schest1}  |G_j(j+1)| \lesssim \dfrac{1}{\sqrt{\beta}}\, \| g_j\|\lesssim \dfrac{1}{\sqrt{\beta}}\, \| g\|,\; j=0,....,N-1\ee
and 
\be\label {schest2}  |(\partial_x G_j)(j+1)| \lesssim  \| g_j\|\lesssim \| g\|,\; j=0,....,N-1.\ee

Therefore using the order (\ref{schorderM}) and estimate (\ref{schest1})-(\ref{schest2}) for $W_k$
in (\ref{schrsecondmembre}) we get 
\be\label{schestomega1} 
\omega_{N,1} \lesssim \dfrac{1}{\sqrt{\beta}}\, \| g\|.
\ee

Now, we remark that for all $j=0,...,N-1,\; \det M_j=1,$ which implies from (\ref{schalphagamma}) that
$$ \alpha_{N,1} \gamma_{N,2}-\alpha_{N,2} \gamma_{N,1}=1.$$
Thus 
$$|\alpha_{N,1}+i \gamma_{N,1}||\alpha_{N,2}+i \gamma_{N,2}|\geq Re[(\alpha_{N,1}+i \gamma_{N,1})(\alpha_{N,2}+i \gamma_{N,2})]=1,$$

implies  with (\ref{schorderM}) and (\ref{schalphagamma}) that
$$\dfrac{1}{ |\alpha_{N,1}+i \gamma_{N,1}|}\leq |\alpha_{N,2}+i \gamma_{N,2}|\leq O(\sqrt{\beta}).$$
The previous estimate and (\ref{schestomega1}) lead to 
\be\label{schfinal1} |c_0^1| \lesssim \|g\| \mbox{ and } |c_0^2| \lesssim \|g\| . \ee 

 {\bf Last step :} Estimate of $u.$ 
 
 First, from (\ref{defGJ}) we have
$$\|G_j\|\lesssim \|g_j\| \lesssim \|g\|, \; j=0,...,N-1.$$

Then, using (\ref{defFj}), (\ref{schorderM}) and (\ref{schfinal1}) we get by iteration the  components of $F_j,j=0,...,N-1$ satisfy :  
$$c_j^1 \lesssim \|g\| \mbox { and } c_j^1 \lesssim \sqrt{\beta} \|g\|.$$     

Consequently, using the two previous estimates in (\ref{defuj}) we directly obtain that the solution of (\ref{schrL2}) satisfies
   
\be \label{schrfinal2} \|u\| \lesssim \|g\|  \;\; ( \beta \rightarrow +\infty  ).\ee 

{\bf Second case : $\beta<0$. } 

If $\beta<0$ the previous procedure doesn't work but fortunately, in this case, 
we can get the estimate (\ref{schrfinal2}) directly. Indeed, multiplying (\ref{schrequation}) by $\dfrac{\overline{u_j}}{i},$ integrating by parts, summing from $j=0$ to $N-1$ and using the boundary-transmision conditions we have 
$$\begin{array}{lll}
 -\beta \ds \sum_{j=0}^{N-1} \int_j^{j+1}|u_j(x)|^2dx  &+&\ds\sum_{j=0}^{N-1} \rho_j \int_j^{j+1}|\partial_x u_j(x)|^2dx \\ 
&+ & i|u_0(0)|^2\\
&=&\dfrac{1}{i} \ds \sum_{j=0}^{N-1} \int_j^{j+1} g_j(x) \overline{u_j(x)} dx.
\end{array} 
$$

Therefore, 
$$ -\beta \sum_{j=0}^{N-1} \int_j^{j+1}|u_j(x)|^2dx \leq \Re[\dfrac{1}{i} \sum_{j=0}^{N-1} \int_j^{j+1} g_j(x) \overline{u_j(x)} dx ] \leq \|u\| \|g\|,  
$$
and we find 
\be \label{schrfinal3} \|u\| \lesssim \|g\|  \;\; ( \beta \rightarrow -\infty  ).\ee 
Finally the result follows from (\ref{schrfinal2})-(\ref{schrfinal3}).
\end{proof}

\end{document}

%% file: maillagecordes.tex
\setlength{\unitlength}{0.240900pt}
\ifx\plotpoint\undefined\newsavebox{\plotpoint}\fi
\begin{picture}(1500,900)(0,0)
\sbox{\plotpoint}{\rule[-0.200pt]{0.400pt}{0.400pt}}%

\put(120,586){\makebox(0,0){$\rho_0$}}
\put(120,245){\makebox(0,0){string 1}}

\put(320,586){\makebox(0,0){$\rho_1$}}
\put(320,245){\makebox(0,0){string 2}}

\put(540,586){\makebox(0,0){$\rho_2$}}
\put(540,245){\makebox(0,0){string 3}}

\put(820,586){\makebox(0,0){$...$}}
\put(820,245){\makebox(0,0){...}}

\put(1120,586){\makebox(0,0){$\rho_{N-2}$}}
\put(1120,245){\makebox(0,0){string N-1}}

\put(1320,586){\makebox(0,0){$\rho_{N-1}$}}
\put(1320,245){\makebox(0,0){string N}}

\put(20,449){\usebox{\plotpoint}}
\put(20,449){\makebox(0,0){$\bullet$}}
\put(220,449){\makebox(0,0){$\bullet$}}
\put(420,449){\makebox(0,0){$\bullet$}}
\put(620,449){\makebox(0,0){$\bullet$}}
\put(820,449){\makebox(0,0){$\bullet$}}
\put(1020,449){\makebox(0,0){$\bullet$}}
\put(1220,449){\makebox(0,0){$\bullet$}}
\put(1420,449){\makebox(0,0){$\bullet$}}
\put(8.0,449.0){\rule[-0.200pt]{341.837pt}{0.400pt}}
\end{picture}